\documentclass[10pt]{amsart}

\usepackage[pagebackref,colorlinks=true]{hyperref}  

\usepackage{amsfonts}
\usepackage{amsmath,amsthm,amssymb,amscd,enumerate,eucal,url}

\usepackage{eucal,url,amssymb,verbatim,
enumerate,amscd,
}

\numberwithin{equation}{section}

\newtheorem{thrm}{Theorem}[section]

\newtheorem{lemma}[thrm]{Lemma}

\newtheorem{prop}[thrm]{Proposition}

\newtheorem{cor}[thrm]{Corollary}

\newtheorem{conv}[thrm]{Convention}

\usepackage[margin=1in]{geometry}
\newcommand{\R}{\mathbb{R}}

\newcommand{\Hnn}{\mathbb{H}^{n+1}}
\newcommand{\QH}{\boldsymbol {G\,(\mathbb{H})}}

\newcommand{\e}{\textbf {e}}

\newcommand{\A}{{A}}
\newcommand{\B}{{B}}
\newcommand{\C}{{\mathbb C}}
\renewcommand{\H}{{\mathbb H}}
\newcommand{\D}{{D}}
\newcommand{\spl}{\mathfrak{ sp}(1)}

\begin{document}

\begin{abstract}
We show that any compact quaternionic contact (abbr. qc)
hypersurfaces in a hyper-K\"ahler manifold which is not totally
umbilical has an induced qc structure, locally qc homothetic to
the standard 3-Sasakian sphere. We also show
that any nowhere umbilical qc hypersurface in a hyper-K\"ahler manifold is endowed with an involutive
7-dimensional distribution whose integral leafs are locally
qc-conformal to the standard 3-Sasakian sphere.
\end{abstract}



\title[Non-umbilical quaternionic contact  hypersurfaces in hyper-K\"ahler manifolds]
{Non-umbilical quaternionic contact  hypersurfaces in hyper-K\"ahler manifolds}

\date{\today}

\author{Stefan Ivanov}
\address[Stefan Ivanov]{University of Sofia, Faculty of Mathematics and Informatics,
blvd. James Bourchier 5, 1164,
Sofia, Bulgaria}
\address{and Institute of Mathematics and Informatics, Bulgarian Academy of
Sciences}
\email{ivanovsp@fmi.uni-sofia.bg}

\author{Ivan Minchev}
\address[Ivan Minchev]{University
of Sofia, Faculty of Mathematics and Informatics, blvd. James Bourchier 5, 1164 Sofia, Bulgaria;
Department of Mathematics and Statistics, Masaryk University, Kotlarska 2, 61137 Brno,
Czech Republic}
\email{minchev@fmi.uni-sofia.bg}

\author{Dimiter Vassilev}
\address[Dimiter Vassilev]{
University of New Mexico\\
Albuquerque, NM 87131}
\email{vassilev@unm.edu}

\maketitle


\setcounter{tocdepth}{2} \tableofcontents

\section{Introduction}

Any real hypersurface in a complex manifold carries a natural CR
structure which in the case of a strictly positive Levi form
endows the surface with a natural pseudo-Hermitian structure. The
goal of this paper is to consider a hyper-K\"ahler manifold and
describe the real hypersurfaces  which carry a natural
quaternionic contact (qc) structure. The concept of a qc structure
was originally introduced by O. Biquard \cite{Biq1} as a model for
the conformal boundary at infinity of the quaternionic hyperbolic
space. According to a result in \cite{Biq1, D}, every real
analytic qc structure is the conformal infinity of a unique
(asymptotically hyperbolic) quaternionic-K\"ahler metric defined
in a neighborhood of the qc structure. Similar to the CR case the
question of embedded quaternionic contact hypersurfaces is a
natural one, but  in contrast to the CR case it imposes a rather
strong conditions on the hypersurface. The situation has the
flavor of the K\"ahler versus the hyper-K\"ahler case. As well
known any complex submanifold of  a K\"ahler manifold is a
K\"ahler manifold and a K\"ahler metric is
locally given by a K\"ahler
potential
. In contrast, a hyper-complex manifold of a
hyper-K\"ahler manifold must be totally geodesic and (in general)
there is no hyper-K\"ahler potential (the structure is rigid).
This suggests that we can expect that there are few quaternionic
contact hypersurfaces in a hyper-K\"ahler manifold.  Indeed, we
showed in \cite{IMV5} that given a connected qc-hypersurface $M$
in the flat quaternion space $\H^{n+1}$, then,  up to a
quternionic affine transformation of $\H^{n+1}$, $M$ is contained
in one of the following three hyperquadrics (the 3-Sasakain
sphere, the hyperboloid and the quaternionic Heisenberg group):
\begin{equation}\label{e:model spaces}
(i) \ \ |q_1|^2+\dots+|q_n|^2 + |p|^2=1,\quad
(ii)\ \  |q_1|^2+\dots+|q_n|^2 - |p|^2=-1 ,\quad (iii)\ \ |q_1|^2+\dots+|q_n|^2 +\R{e}(p)=0,
\end{equation}
 where $(q_1,q_2,\dots q_n,p)$ denote the standard quaternionic coordinates of $\Hnn$. We recall that the above three examples are locally qc-conformal.  Furthermore, it was shown  \cite{IMV5} in the general hyper-K\"ahler case  that the Riemannian curvature of the ambient space has to be degenerate along the normal to the qc-hypersurface vector field.

The notion of qc-hypersurface was first defined by Duchemin \cite{D1} in the general setting of quaternionic manifold. A manifold $K$ is
called {\it quaternionic} if $K$ is  endowed with  a 3-dimensional
sub-bundle $\mathcal Q^K\subset End(TK)$ locally generated by a
pointwise quaternionic structure $J_1,J_2,J_3$ together with   a
torsion free connection that preserves $\mathcal Q^K$.

An embedding  $\iota :M\rightarrow K$ of a qc manifold $M$ with a horizontal space $H$ equipped with a quaternion structure $\mathcal Q^H$
, see Section \ref{QC-manifolds} for precise definition,  into a quaternionic manifold $(K,\mathcal
Q^K)$ is called a {\it qc embedding} if the differential $\iota_*$
intertwines  $\mathcal Q^K$ and $\mathcal Q^H$, i.e., if
\begin{equation}\nonumber
\mathcal Q^H=\iota_*^{-1}\, \mathcal Q^K\,\iota_*
\end{equation}
is satisfied at each point of $M$, where $\mathcal Q^H$ denotes
the  point-wise quaternionic structure of the horizontal
distribution $H\subset TM$. In particular, the
image $\iota_*(H)$  coincides with the maximal $\mathcal
Q^K-$invariant subspace of $\iota_*(TM)\subset TK$.  A real
hypersurface  $M\subset K$ in a quaternionic manifold $K$  is
called a {\it qc hypersurface} if there exists a qc structure on
$M$ for which the inclusion map is a qc embedding. Notice that, if
such a qc structure exists, then it is unique, since the qc
distribution $H$  is the
maximal $\mathcal Q^K$ invariant subspace of $TM$.

Duchemin \cite{D1}showed that a real analytic qc manifold can be realized as a qc-hypersurface in an appropriate quaternionic manifold.

In this paper we consider qc-hypersurfaces in a hyper-K\"ahler manifold. Our main result in the case of a compact embedded qc-hypersurface is the following.
\begin{thrm}\label{t:main1}
Let $M$ be a compact  qc-hypersurface of a hyper-K\"ahler
manifold. If $M$ is not a totally umbilical hypersurface, then
the qc-conformal class of the embedded qc structure contains a
qc-Einstein structure of positive qc-scalar curvature which is
locally 
qc-equivalent to the 3-Sasakian sphere.
\end{thrm}
We note that the existence of a conformal factor leading to a
qc-Einstein structure, called calibrated qc-structure,  was
established earlier by the authors, see \cite[Theorem 1.2]{IMV5}.
Thus, the main new result here is the qc-conformal flatness of the
calibrated qc-Einstein structure. In the connected
simply-connected case the above Theorem implies that the
qc-conformal class of the embedded qc structure contains a qc
structure 
qc-equivalent to the round 3-Sasakian sphere, see also Theorem
\ref{main12}. It is well known that any totally umbilical
hypersurface of a hyper-K\"ahler manifold is a qc-hypersurface
whose qc structure is generated by its induced 3-Sasakian metric.
Furthermore, a 3-Sasakian space can be embedded as a totally
umbilical qc-hypersurface in a hyper-K\"ahler manifold, namely in
its metric cone. The hyperquadric
$$|q_1|^2+\dots+|q_n|^2 + 2|p|^2=1$$ in $\Hnn$ is an example of a
compact qc-hypersurface which is not totally umbilical with respect to the standard flat hyper-K\"ahler metric of $\Hnn$.

The case of a local qc-embedding is considered in Section \ref{s:loc emb} where we prove results which  in the seven dimensional case give  the following theorem.
\begin{thrm}\label{t:main2}
A seven dimensional everywhere non-umbilical qc-hypersurface $M$
embedded in a hyper-K\"ahler manifold is qc-conformal to a
qc-Einstein structure which is locally
qc-equivalent to  the 3-Sasakian sphere, the quaternionic
Heisenberg group or the hyperboloid.
\end{thrm}

The proofs of the main results rely on the known and some new properties of the "calibratng" qc-conformal factor. More precisely, as shown in \cite{IMV5}, given a qc-hypersurface $M$ in a hyper-K\"ahler manifold $K$ there is a positive function $f$ on $M$ called "calibrating" function so that the qc structure on $M$ obtained from the embedded one with $f$ as a qc-conformal factor is qc-Einstein, see \cite[Lemma 3.7]{IMV5}. Furthermore, if  $II$ is  the second fundamental form of $M$,  then the (0,2) tensor $f\,II$ extends to a covariant constant along $M$, see \cite[Theorem 3.1]{IMV5}. The new key points for the results of the current paper are certain identities for  the second and third order (horizontal) covariant derivative of the calibrating function $f$. Using the bracket generating condition and the relation between the Biquard and Levi-Civita connections these identities lead to a  third order differential system on $M$ well studied in the Riemannian case by several authors, see \cite{Obata, Galo,Tan, Matv}.  In the compact case, this system is known to have the remarkable property that  it admits a non-constant solution only on Riemannian manifolds which are locally isometric to the round sphere.

\begin{conv}\label{conven}

Throughout the paper, unless explicitly stated otherwise, we will
use the following notation.

\begin{enumerate}[a)]
\item All manifolds are assumed to be $C^\infty$ and connected.
\item The triple $(i,j,k)$ denotes any positive permutation of
$(1,2,3)$.
\item $s,t$ are any numbers from the set $\{1,2,3\}$, $s,t \in \{1,2,3\}$.
\item For a given decomposition $TM=V\oplus H$ we denote by
$[.]_V$ and $[.]_H$ the corresponding projections to $V$ and $H$.
\item ${A}, {B}, {C}$, etc. will denote
sections of the tangent bundle of $M$,  ${A},
{B}, {C}\in  TM$.
\item $X,Y,Z,U$ will denote horizontal vector fields,
$X,Y,Z,U\in H$.
\end{enumerate}
\end{conv}

{\bf Acknowledgments.} S.I. and I.M. are partially supported by
Contract DFNI I02/4/12.12.2014 and  Contract 195/2016 with the
Sofia University "St.Kl.Ohridski". I.M. is supported by a SoMoPro
II Fellowship which is co-funded by the European
Commission\footnote{This article reflects only the author's views
and the EU is not liable for any use that may be made of the
information contained therein.} from \lq\lq{}People\rq\rq{}
specific programme (Marie Curie Actions) within the EU Seventh
Framework Programme on the basis of the grant agreement REA No.
291782. It is further co-financed by the South-Moravian Region. DV
was partially supported by Simons Foundation grant \#279381. 
The authors would like to
thank the Masaryk University, Brno, for the hospitality and
the financial support provided while visiting the Department of
Mathematics.

\section{Preliminaries}
\subsection{Quaternionic contact manifolds}\label{QC-manifolds}

 Here, we recall briefly the relevant facts and notation needed for this paper and refer to \cite{Biq1}, \cite{IMV1} and \cite{IV3} for a more detailed exposition. A quaternionic contact (qc) manifold is a $(4n+3)$-dimensional manifold  $M$ with a codimension three distribution $H$  equipped with  an $Sp(n)Sp(1)$ structure locally defined by an $\mathbb{R}^3$-valued 1-form $\eta=(\eta_1,\eta_2,\eta_3)$. Thus, $H=\cap_{s=1}^3 Ker\, \eta_s$
carries a positive definite symmetric tensor $g$, called the horizontal metric, and a compatible rank-three bundle $\mathcal {Q}^H$
consisting of endomorphisms of $H$ locally generated by three orthogonal almost complex
structures $I_s$,  satisfying the unit quaternion relations: (i) $I_1I_2=-I_2I_1=I_3, \quad $ $I_1I_2I_3=-id_{|_H}$; \hskip.1in (ii) $g(I_s.,I_s.)=g(.,.)$; and \hskip.1in  (iii) the
compatibility conditions  $2g(I_sX,Y)\ =\ d\eta_s(X,Y)$, $
X,Y\in H$  hold true. Unlike the CR case, in the qc case the horizontal space determines uniquely the qc-conformal class, cf. \cite{IMV5}.  For this reason very often we will identify the qc structure with the   $\mathbb{R}^3$-valued 1-form $\eta$ while supressing the remaining data. We also note that by virtue of its definition a quaternionic contact manifold is orientable.

Two qc structures $\eta$ and $\bar\eta$ on a manifold $M$ are called \emph{qc-conformal} to each other if  $\bar\eta=\mu\Psi\eta$ for a positive smooth
function $\mu$ and an $SO(3)$ matrix $\Psi$ with smooth functions as
entries.   A diffeomorphism $F$ between two qc manifolds $M$ and $\bar M$ is  called \emph{quaternionic contact conformal (qc-conformal) transformation}  if $F*\bar\eta=\mu\Psi\eta$ .  The qc-conformal curvature tensor $W^{qc}$, introduced in \cite{IV1}, is the
obstruction for a qc structure to be locally qc-conformally  to the
standard 3-Sasakian structure on the $(4n+3)$-dimensional sphere \cite{IV1,IV3}.  As already noted in the introduction the 3-Sasakain sphere, the hyperboloid and the quaternionic Heisenberg group are all locally qc-conformal to each other.

As shown in \cite{Biq1}, there is a "canonical" connection associated to every qc manifold of dimension at least eleven. In the seven dimensional case the existence of such a connection requires the qc structure to be integrable \cite{D}. The integrability condition is equivalent to the existence of Reeb vector fields \cite{D}, which (locally) generate the supplementary to $H$ distribution $V$. The Reeb vector fields $%
\{\xi_1,\xi_2,\xi_3\}$ are determined by \cite{Biq1}
\begin{equation}  \label{bi1}
 \eta_s(\xi_t)=\delta_{st}, \qquad (\xi_s\lrcorner
d\eta_s)_{|H}=0,\quad (\xi_s\lrcorner d\eta_t)_{|H}=-(\xi_t\lrcorner
d\eta_s)_{|H},
\end{equation}
where $\lrcorner$ denotes the interior multiplication.  Henceforth, by a qc structure in dimension $7$ we shall mean a qc structure satisfying \eqref{bi1} and refer to the "canonical" connection as \emph{the Biquard connection}. The Biquard connection is the unique linear connection preserving the decomposition $TM=H\oplus V$ and the $Sp(n)Sp(1)$ structure on $H$ with torsion $T$ determined by $T(X,Y)=-[X,Y]_{|_V}$ while  the endomorphisms $T({\xi_s},.): H \rightarrow H$ belong to the orthogonal complement $(sp(n)+sp(1))^{\perp}\subset GL(4n,R)$.

The covariant derivatives with respect to the Biquard connection of
the endomorphisms $I_s$  and the Reeb vector fields are given by
\begin{equation}\label{der}
\nabla I_i=-\alpha_j\otimes I_k+\alpha_k\otimes I_j,\qquad
\nabla\xi_i=-\alpha_j\otimes\xi_k+\alpha_k\otimes\xi_j.
\end{equation}
 The $\spl$-connection 1-forms  $\alpha_1,\alpha_2, \alpha_3$, defined by the
above equations satisfy \cite{Biq1}
\begin{equation*}
\alpha_i(X)=d\eta_k(\xi_j,X)=-d\eta_j(\xi_k,X),\qquad X\in H.
\end{equation*}

Let $R=[\nabla,\nabla]-\nabla_{[.,.]}$ be the curvature tensor of
$\nabla$ and $R(\A,\B,\C,\D)=g(R_{\A,\B}\C,\D)$ be the
corresponding curvature tensor of type (0,4). The qc Ricci tensor
$Ric$, the qc-Ricci forms $\rho_s$ and the normalized qc scalar
curvature $S$ are defined by
\begin{equation}\nonumber
Ric(\A,\B)=\sum_{a=1}^{4n}R(e_a,\A,\B,e_a), \quad 4n\rho_s(A,B)=\sum_{a=1}^{4n}R(A,B,e_a,I_se_a), \quad 8n(n+2)S=Scal=\sum_{a=1}^{4n}Ric(e_a,e_a),
\end{equation}
where $e_1,\dots,\e_{4n}$ of $H$ is an $g$-orthonormal frame on $H$.

We say that $(M,\eta)$ is a qc-Einstein manifold if the restriction of the qc-Ricci tensor to the horizontal space $H$ is  trace-free, i.e.,
$$Ric(X,Y)=\frac{Scal}{4n}g(X,Y)=2(n+2)Sg(X,Y), \quad X,Y\in H.$$
The qc-Einstein condition is equivalent to the vanishing of the
torsion endomorphism of the Biquard connection, $T(\xi_s,X)=0$
\cite{IMV1}.  It is also known \cite{IMV1,IMV4}  that the
qc-scalar curvature of a qc Einstein manifold is  constant and the
vertical distribution is integrable.

 By \cite[Theorem 5.1]{IMV4}, see also \cite{IV2} and \cite[Theorem 4.4.4]{IV3} for alternative proofs in the case $Scal\not=0$, a qc-Einstein structure is characterised by either of the following equivalent conditions
\begin{enumerate}[i)]
\item  locally, the given qc structure is defined by  1-form $(\eta_1,\eta_2,\eta_3)$  such that for some constant $S$, we have
\begin{equation}\label{str_eq_mod}
d\eta_i=2\omega_i+S\eta_j\wedge\eta_k;
\end{equation}
\item locally,  the given qc structure is defined by a 1-form $(\eta_1,\eta_2,\eta_3)$  such that the corresponding connection 1-forms vanish on $H$ and (cf. the proof of Lemma 4.18 of \cite{IMV1})
\begin{equation}\label{e:vanishing alphas qc-Einstein}
\alpha_s=-S\eta_s.
\end{equation}
\end{enumerate}

\subsubsection{The correspondin (pseudo) Riemannian geometry }

Let  $M$ be a qc-Einstein manifold.  Note that, by applying an
appropriate qc homothetic transformation, we can aways reduce a
general qc-Einstein structure to one whose normalized qc-scalar
curvature $S$ equals $0,2$ or $-2$. Consider the one-parameter
family of (pseudo)  Riemannian metrics  $h^{\mu},\ \mu\ne 0$ on
$M$ by letting
\begin{equation}\label{mh} h^{\mu}=g|_H+\mu(\eta_1^2+\eta^2_2+\eta_3^2).
\end{equation}
Let $\nabla^\mu$ be the Levi-Civita connection of $h^\mu.$ Note
that $h^{\mu}$ is a positive-definite metric when $\mu>0$ and has
signature $(4n,3)$ when $\mu<0$. The difference
$L=\nabla^{\mu}-\nabla$  between the Levi-Cevita connection
$\nabla^{\mu}$ and the Biquard connection $\nabla$ is given by
\cite{IMV1,IMV4}
\begin{equation}\label{nabla-lambda}
L(A,B)\equiv \nabla^{\mu}_{\A}\B-\nabla_{\A}\B=
\frac{S}{2}[\A]_V\times [\B]_V +
\sum_{s=1}^3\Big\{-\omega_s(\A,\B)\xi_s+\mu\eta_s(\A)I_s\B+\mu\eta_s(\B)I_s\A
\Big\},
\end{equation}
where $._V\times._V$ is the standard vector cross product on the 3-dimensional vertical space $V$.

\subsection{Quaternionic contact  hypersurfaces}\label{ss:qc-hypersurfaces}
In this section we summarize some results from \cite{IMV5} which are the starting point of the subject of the current paper. For ease of reading we follow \cite{IMV5} closely.

\subsubsection{qc-hypersurfaces}\label{ss:qc-hypersurf} Let $K$ be a  hyper-K\"ahler manifold with hyper-complex structure
$(J_1,J_2,J_3)$, quaternionic bundle $\mathcal{Q}^K$, and
hyper-K\"ahler metric  $G$. In particular, the Levi-Civita
connection $D$ is a torsion free connection on $K$
preserving $\mathcal{Q}^K$.

For a real hypersurface $M\subset K$ the maximal
$\mathcal{Q}^K$-invariant subspace  $TM$ is denoted by $H$ and refereed to as the horizontal   
distribtution. If $\iota: M\rightarrow K$ is the natural inclusion
map, then $M$ is a qc-hypersurface if it is a qc manifold with
respect to the induced quaternionic structure
$\iota_*^{-1}(Q^K)\iota_*$ on $H$. In order to simplify the
notation we shell identify the corresponding points and tensor
fields on $M$ with their images through $\iota$ in $K$. An
equivalent characterization of a qc-hypersurface $M$ is that the
restriction of the second fundamental form of $M$ to the
horizontal space $H$ is a definite symmetric form, which is
invariant with respect to the induced quaternion structure, see
\cite[Proposition~2.1]{D1}. After choosing the unit normal vector
$N$  to $M$ appropriately, we will assume that the second
fundamental form of $M$,  $$II(A,B) =-G(D_AN,B), \quad A,B\in
TM,$$ is negative definite on  the horizontal space $H$. The
defining tensors of the embedded qc structure on $M$ are  given by
\begin{equation}\label{e:def of hat str}
\hat\eta_s (A)= G(J_sN,A), \quad \hat \xi_s=J_s N +\hat r_s, \quad \hat\omega_s(X,Y)=-II(I_sX,Y), \quad \hat g(X,Y)= -\hat\omega_s(I_sX,Y),
\end{equation}
where  $I_s=J_s|_H$ and $\hat\xi_s$, are the Reeb vector fields corresponding to $\hat\eta_s$, see \cite[Section 2.2]{IMV5}.

\subsubsection{The calibrating function}

Let $\hat\omega_s$ be the fundamental 2-forms corresponding to $\hat\eta_s$,  given  by $2\hat\omega_s(X,Y)=d\hat\eta_s(X,Y),\ X,Y\in H$ and $\hat\xi_t\lrcorner\hat\omega_s=0,\ s,t=1,2,3.$ Following \cite[Section 3.1]{IMV5}, consider the complex 2-forms on $M$,
\begin{equation*}
\begin{aligned}
\hat\gamma_i=\hat\omega_j+\sqrt{-1}\ \hat\omega_k,\qquad
{\Gamma}_i(\A,\B)\ = \ G(J_j\A,\B)+\sqrt{-1}\, G(J_k\A,\B).
\end{aligned}
\end{equation*}
Using a type decomposition argument it was shown in \cite[Section 3.1]{IMV5} that
\begin{equation}\label{e:def pf mu}
{\Gamma}_s^n \equiv \mu_s \hat\gamma_s^n \mod \{\hat\eta_1,\hat\eta_2,\hat\eta_3\},
 \end{equation}
for $s=1,2,3$ and some complex valued functions $\mu_s$ and, in fact,
$\mu_1=\mu_2=\mu_3=\mu$ for a  positive (real valued)
function  $\mu$  on $M$. The \emph{calibrating} function of $M$ was
defined by
$$f=\mu^{\frac{1}{n+2}}.$$

\subsubsection{The calibrated qc structure}

The qc structure
\begin{equation}\nonumber
(\eta_1,\eta_2,\eta_3)\overset{def}{=}f(\hat\eta_1,\hat\eta_2,\hat\eta_3)
\end{equation}
is called \emph{calibrated}. As shown in \cite{IMV5}, it satisfies the structure equations \eqref{str_eq_mod}. In particular, it is a qc-Einstein structure. Moreover, by \cite[Lemma 3.9]{IMV5} the horizontal metric $g$ of the calibrated  qc structure is related to the second fundamental form of the qc-embedding by the formula
\begin{equation}\label{II-form}
g(\A '',\B '')=-fII(\A,\B)-\frac {S}{2}\sum_{s=1}^3\eta_s(\A)\eta_s(\B),\quad  \A,\B\in TM,
\end{equation}
where for $\A\in TM$ we let $\A ''=A-\sum_{s=1}^3\eta_s(\A)\xi_s$ be the horizontal part of $\A$. The corresponding Reeb vector fields $\xi_s$ are given  by
\begin{equation}\label{xi-JN}
\xi_s=J_s\Big(f^{-1}N+r\Big),
\end{equation}
where $r\in H$ is determined by $II(r,X)={f}^{-2}df(X),\ X\in H$. In fact, we have \cite[ Lemma 3.8]{IMV5}
\begin{align}\label{r}
& r=-f^{-1}\nabla f,\\\label{e:df xi_s}
& df(\xi_s)=0, \qquad s=1,2,3,
\end{align}
where $\nabla f\in H$ denotes the horizontal gradient of $f$,
$df(X)=g(\nabla f,X)$.

The  \emph{calibrated transversal} to $M$ vector field is defined by
\begin{equation}\label{e:cal transversal}
\xi{=}f^{-1}N+r.
\end{equation}
From \eqref{xi-JN} and \eqref{e:cal transversal} we have
 \begin{equation}\label{e:xi_s by xi}
 \xi_s =J_s \xi.
 \end{equation}

With the obvious identifications, the bundle $TK|_M\rightarrow M$  decomposes as a direct sum,
\begin{equation}\label{decompose}
TK|_M=H\oplus V\oplus{\mathbb R}\xi,
\end{equation}
where $V$ is the span of the Reeb vector fields $\xi_s$ of the calibrated qc structure on $M$. For  $v\in T_pK$ we define
 \begin{equation}\label{e:vector decomp '}
 v'=v-\lambda(v)\xi(p)\in T_p M=H_p\oplus V_p,\qquad v''=\pi v =v{'}-\sum_{s=1}^3\eta_s(v')\xi_s\in H_p,
 \end{equation}
  where  $\lambda=fG(N,.)$ so that $v'$ is the projection of $v$ on $T_p M=H_p\oplus V_p$ parallel to the calibrated transversal field $\xi$ and  $\pi:TK|_M\rightarrow H$ is the projection on the horizontal space using the decomposition \eqref{decompose}.
Thus, for $v\in TK|_M$ we have
\begin{equation}\label{e:lambda by eta}
\lambda(J_sv)=\eta_s(v')
\end{equation}
and the decomposition
\begin{equation}\label{e:vector decomp}
v=\pi v+ \sum_{s=1}^3\eta_s(v')\xi_s+\lambda(v)\xi\, \in H\oplus V\oplus{\mathbb R}\xi.
\end{equation}
Following \cite[(3.23)]{IMV5} consider the symmetric bilinear form $\mathfrak{W}\in T^*K|_M\otimes T^*K|_M,$
\begin{multline}\label{delta-def}
{\mathfrak{W}}(v,w)\overset{def}{=} -fII(v',w')+ \frac {S}{2}\lambda(v)\lambda(w)=
g(\pi v,\pi w)+ \frac {S}{2}\sum_{s=1}^3\eta_s(v')\eta_s(w') +
\frac {S}{2}\lambda(v)\lambda(w).
\end{multline}
Clearly, $\mathfrak{W}(J_s.,.J_s.)=\mathfrak{W}(.,.)$, $s=1,2,3,$ and 
$\mathfrak{W}$ as the unique
$J_s$-invariant extension of the symmetric bilinear form $-f II$
on $TM$ to a symmetric bilinear form on $TK|_M$. A very important
property of the calibrated qc structure is that $\mathfrak{W}$
is constant along $M$ with respect to the Levi-Civita connection
$D$ of the hyper-K\"ahler metric $G$, see  \cite[Theorem 3.1]{IMV5}), i.e.,
we have
\begin{equation}\label{parallel}
D_A\mathfrak{W}=0,\quad A\in TM.
\end{equation}

Finally, we record an important relation between the calibrating function and the parallel bilinear form, see \cite[(2.16)]{IMV5}
\begin{equation}\label{df}
{\mathfrak{W}}(N,\A)=-fII(N',\A)=f^2II(r,\A)=-fg(r,\A'')=df(\A'')=df(\A).
\end{equation}

\section{The system of differential equations for the calibrating function}

We begin with a lemma relating the
Levi-Civita connection $D$ of the hyper-K\"ahler metric $G$ to
the Biquard connection $\nabla$ of  the calibrated qc structure on
$M$.

\begin{lemma} \label{D-nabla}
For any $A\in TM$ and $X\in H$ we have:

\par i) $D_A X=\nabla_A X + \sum_{t=1}^3\Big((S/2)\eta_t(A) I_tX - \omega_t(\pi A,X)\xi_t\Big) - g(\pi A,X)\xi$.
\par ii) $D_A \xi = (S/2) A$ and $D_A \xi_s = (S/2) J_sA$.
\end{lemma}
\begin{proof}
First we shall prove the formula in part \textit{i) } for a horizontal vector field $A$,
\begin{equation}\label{DXY}
D_XY
=\nabla_XY-\omega_s(X,Y)\xi_s-g(X,Y)\xi.
\end{equation}
We start with the computation of the horizontal part of $D_XY$,
\begin{equation}\label{piD}
\nabla_XY=\pi(D_XY),\qquad X,Y\in H,
\end{equation}
recalling that $\pi$ is the projection on the horizontal space, see \eqref{e:vector decomp}.
From \eqref{delta-def} and \eqref{parallel} we have
\begin{multline*}
0 = (D_X\mathfrak{W})(Y,Z)
=X\Big(\mathfrak{W}(Y,Z)\Big)-\mathfrak{W}(D_XY,Z)-\mathfrak{W}(Y,D_XZ)\\
=X\Big(g(Y,Z)\Big)-g\Big(\pi(D_XY),Z\Big)-g\Big(Y,\pi(D_XZ)\Big).
\end{multline*}
Letting $F(X,Y)\overset{def}=\nabla_XY-\pi(D_XY),$
we compute
\begin{multline*}
0 = (\nabla_X g)(Y,Z)
=X\Big(g(Y,Z)\Big)-g\Big(\pi(D_XY)+F(X,Y),Z\Big)-g\Big(Y,\pi(D_XZ)+F(X,Z)\Big)\\
=-g\Big(F(X,Y),Z\Big)-g\Big(F(X,Z),Y\Big),
\end{multline*}
while on the other hand
\begin{multline*}
0 = g\Big(\pi (T(X,Y)),Z\Big)=g\Big(\nabla_XY-\nabla_YX-\pi([X,Y]),Z\Big) = g\Big(\nabla_XY-\nabla_YX-\pi(D_XY-D_YX),Z\Big)\\
=g\Big(F(X,Y),Z\Big)-g\Big(F(Y,X),Z\Big).
\end{multline*}
Thus, the tensor $g\Big(F(X,Y),Z\Big)$  is both symmetric in $X,Y$ and skew-symmetric in $Y, Z$ which implies that it vanishes.

The remaining part of $D_XY$  in the decomposition based on  \eqref{e:vector decomp} can be computed easily as follows,
\begin{equation*}
\lambda(D_XY)
=-fG(D_XN,Y)=f II(X,Y)=-g(X,Y)
\end{equation*}
and
\begin{equation*}
\eta_s((D_XY)')=-\lambda(J_sD_X Y)=-\lambda(D_X (J_sY))=g(X,J_sY)=-\omega_s(X,Y).
\end{equation*}
From the above the formula in part \textit{i)}
in the case when $A$ is a horizontal vector field follows.

Next we prove the formula
\begin{equation}\label{DXN}
D_XN=\nabla_X\nabla f+\frac{Sf}{2}X-df(J_sX)\xi_s.
\end{equation}
In order to determine the horizontal part of $D_XN$ we recall \eqref{parallel} and then compute the (horizontal) Hessian of $f$ as follows
\begin{multline*}
\nabla^2 f(X,Y)= X(df(Y))-df(\nabla_XY)= X(\mathfrak{W}(N,Y))- df(\nabla_XY)\\
= \mathfrak{W}(D_XN,Y)+ \mathfrak{W}(N,D_X Y)-df(\nabla_XY)=\mathfrak{W}(D_XN,Y)+ \mathfrak{W}(N,D_X Y-\nabla_XY)
\end{multline*}
using \eqref{df} in the last equality. From \eqref{delta-def} and \eqref{DXY} it follows
\[
\nabla^2 f(X,Y)= g(\pi(D_X N),Y)-\frac {1}{2}fS g(X,Y)
\]
noting that $\mathfrak{W}(N,\xi)=\frac 12 fS$.
The vertical part of $D_X N$ is computed with the the help of \eqref{II-form} and \eqref{e:cal transversal}
\[
\eta_s(D_X N)=-f G\Big(N,J_s(D_X N),\Big)=f G\Big(D_X N,J_sN\Big)=-f II(X,J_sN)
=-df(J_s X).
\]
The proof of formula \eqref{DXN} is complete.

An immediate consequence of \eqref{e:cal transversal}, \eqref{r}, \eqref{DXY} and \eqref{DXN} is the following formula
\begin{equation}\label{DXxi}
D_X\xi\overset{\eqref{xi-JN}}=\frac{1}{2}S X.
\end{equation}
At this point we can complete the proof of part \textit{i)}.
Since the calibrated qc structure is qc-Einstein and  the 1-forms $\eta_s$ satisfy the structure equations \eqref{str_eq_mod}, we have  $\nabla_{\xi_s} X=[\xi_s,X]$. Therefore,
\begin{equation}\label{DxiX}
\nabla_{\xi_s} X=[\xi,X]=D_{\xi_s}X-D_X\xi_s
=D_{\xi_s}X-J_s(D_X\xi)\overset{\eqref{DXxi}}=D_{\xi_s}X-\frac{S}{2}J_sX.
\end{equation}
Finally, we compute
\begin{multline*}
D_AX=D_{\pi A}X+\eta_s(A)D_{\xi_s}X\overset{\eqref{DXY},\eqref{DxiX}}=\nabla_{\pi A}X-\omega_s(\pi A,X)\xi_s-g(\pi A,X)\xi +\eta_s(A)\Big(\nabla_{\xi_s}X+\frac{S}{2}J_sX\Big)\\
=\nabla_A X-\omega_s(\pi A,X)\xi_s -g(\pi A,X)\xi+\frac{S}{2}\eta_s(A)J_sX.
\end{multline*}

Turning to the proof of \textit{ii)}, we have from \eqref{II-form} and \eqref{e:lambda by eta} the formula
\[
G(D_{\xi_s}
N,A)=-II(\xi_s,A)=\frac{1}{2}f S\eta_s(A)=\frac{S}{2}G(J_sN,A),
\]
hence
\begin{equation}\label{DxiN}
 D_{\xi_s} N = \frac{1}{2}S J_s N=\frac{1}{2}S\xi_s-\frac 12S J_s\nabla r.
\end{equation}

From \eqref{e:df xi_s} and $T(\xi_s,X)=0$ it follows $\nabla_{\xi_s}\nabla f=0$, hence
\eqref{e:cal transversal}, \eqref{r}, \eqref{DxiN} and \eqref{xi-JN}  give
\[
D_{\xi_s}\xi=\frac{S}{2}\xi_s,
\]
which together with \eqref{DXxi} completes the proof of part \textit{ii}) after recalling \eqref{e:xi_s by xi}.
\end{proof}

\begin{cor}\label{l:umbilic}
$M$ is a totally umbilical  qc-hypersurface of a hyper-K\"ahler
manifold iff the calibrating function is locally constant.
\end{cor}
\begin{proof}
In view of \eqref{DxiN} and \eqref{r} it follows the horizontal gradient of $f$ vanishes $\nabla f=0$, hence $f$ is locally constant taking into account that the horizontal space is bracket generating.
\end{proof}

As customary,  let ${\mathfrak{W}}:M\rightarrow End(TK)|_M$ also denote the (1,1) tensor corresponding to the symmetric  bilinear form $\mathfrak{W}$, i.e.,  $G({\mathfrak{W}}u,v)=\mathfrak{W}(u,v)$ for all $u,v\in TK|_M$.
Then ${\mathfrak{W}} J_s=J_s {\mathfrak{W}}$ and, since both $G$ and
$\mathfrak{W}$ are $D$-parallel along $M$, we also have
\begin{equation}\label{D-par}
(D_A{\mathfrak{W}})(u)=0,\qquad A\in TM,\ u\in TK|_M.
\end{equation}

An almost immediate corollary of the proof of Lemma \ref{D-nabla} is the following formula for $\mathfrak{W}$ in terms of the calibrating function.
\begin{lemma}\label{Delta} For $X\in H$ we have:
\par i) ${\mathfrak{W}}X=f\nabla_X\nabla f+(Sf^2/2) X + df(X) \nabla f - f\sum_{s=1}^3df(I_sX)\xi_s+fdf(X)\xi;$
\par ii) ${\mathfrak{W}}\xi=(Sf/2)\nabla f+(Sf^2/2)\xi;$
\par iii) ${\mathfrak{W}}\xi_s=(Sf/2)I_s\nabla f+(Sf^2/2)\xi_s,$ $s=1,2,3$.
\end{lemma}
\begin{proof}
By definition \eqref{delta-def}, recall also \eqref{e:vector decomp '}, we have
\begin{multline*}
\mathfrak{W}(X,u)=-f II (X,u')= f G(D_XN,u')=fG(D_XN,u)-fG(D_XN,\xi)\lambda(u)\\
= f G(D_XN,u')=fG(D_XN,u)-f^2G(D_XN,\xi)G(N,u).
\end{multline*}
Now, the formula of part\textit{ i)}  follows by a direct substitution using \eqref{e:cal transversal}, \eqref{r} and \eqref{DXN}. Finally, part iii) follows from $J_s\xi=\xi_s$, see after equation \eqref{e:cal transversal}.

Part{\textit{ii)} is proved similarly with the help of \eqref{DxiN} instead of \eqref{DXN}}.
\end{proof}

After the preceding technical lemmas we turn to the key result which gives a system of partial differential equations for the calibrating function. With the help of \eqref{nabla-lambda}  is then expressed in terms of Levi-Civita connection in the subsequent lemma.
\begin{lemma}\label{l:key lemma}
The function $\phi\overset{def}=\frac{1}{2}f^2$ satisfies the following equations
\begin{align}\label{eqbi3}
& d\phi(\xi_s)=0;\\\label{eqbi2}
&\nabla^2\phi(X,Y)=\nabla^2\phi(I_sX,I_sY);
\\\label{eqbi1}
& \nabla^3\phi(X,Y,Z)+Sd\phi(X)g(Y,Z)+\frac{S}2d\phi(Y)g(Z,X)
              +\frac{S}2d\phi(Z)g(X,Y)\\\nonumber
&\hskip3in=\frac{S}2\sum_{s=1}^3\Big[d\phi(I_sY)\omega_s(X,Z)
                              +d\phi(I_sZ)\omega_s(X,Y)\Big].
\end{align}
\end{lemma}
\begin{proof}
Since $
d\phi = fdf$ and $\nabla^2\phi = f\nabla^2f+df\otimes df,
$ 
  \eqref{df} gives $\eqref{eqbi3}$. Recalling the decomposition  \eqref{e:vector decomp}, see also \eqref{e:vector decomp '}, by
Lemma~\ref{Delta} we have
\begin{equation}\label{n2phi}
g\Big(({\mathfrak{W}}X)'',Y\Big)=\nabla^2\phi(X,Y)+S\phi g(X,Y).
\end{equation}
From ${\mathfrak{W}} J_s= J_s {\mathfrak{W}}$ and $g(I_sX,I_sY)=g(X,Y)$ for
 $s=1,2,3$, \eqref{eqbi2} follows from
\begin{equation*}
0=g\Big(({\mathfrak{W}}X)'',Y\Big)-g\Big(({\mathfrak{W}}I_sX)'',I_sY\Big)
=\nabla^2\phi(X,Y)-\nabla^2\phi(I_sX,I_sY).
\end{equation*}
We turn to the proof of \eqref{eqbi1}.  A differentiation of \eqref{n2phi}  gives
\begin{equation}\label{e:nabla^3 phi}
\nabla^3\phi(X,Y,Z)+Sd\phi(X)g(Y,Z)
=g\left (\nabla_X\left({\mathfrak{W}}Y \right)'',Z \right)-g\left(\pi {\mathfrak{W}}\nabla_XY ,Z \right).
\end{equation}
Taking into account \eqref{piD},  \eqref{e:vector decomp} and  \eqref{D-par} we can rewrite the first term in the right-hand side of the above identity as follows
\begin{multline*}
g\left (\nabla_X\left({\mathfrak{W}}Y \right)'',Z \right)=g\left (\pi D_X\left({\mathfrak{W}}Y \right)'',Z \right)
=g\left (\pi WD_X Y ,Z \right)- \eta_s\left( \left({\mathfrak{W}}Y\right)'\right)g\left (\pi D_X\xi_s ,Z \right)-\lambda({\mathfrak{W}}Y)g\left (\pi D_X\xi ,Z \right).
\end{multline*}
Now we use Lemma \ref{Delta} to compute
\[
\eta_s\left( \left({\mathfrak{W}}Y\right)'\right)=-fdf(I_sY)\ \text{ and }\
\lambda({\mathfrak{W}}Y)=fdf(Y).
\]
A substitution of the last two equations in \eqref{e:nabla^3 phi} gives
\begin{multline*}
\nabla^3\phi(X,Y,Z)+Sd\phi(X)g(Y,Z)
=g\left(\pi {\mathfrak{W}} \left(D_X Y-\nabla_XY \right),Z \right)+\frac{Sf}{2}\omega_s(X,Z)df(I_sY)-\frac{Sf}{2}g(X,Z)df(Y)\\
=\frac{S}{2}\Big(\omega_s(X,Y)d\phi(I_sZ)
          -g(X,Y)d\phi(Z)+\omega_s(X,Z)d\phi(I_sY)-g(X,Z)d\phi(Y)\Big)
\end{multline*}
using Lemma \ref{D-nabla} and Lemma \ref{Delta} in the last equality.
The proof of Lemma \ref{l:key lemma} is complete.
\end{proof}

We continue with our main
technical result, which allows the partial reduction to a Riemannian geometry problem.
\begin{prop}
Let $(M,\eta,Q)$ be a (4n+3)-dimensional qc-Einstein space with
constant qc-scalar curvature $S\not=0$ and $\phi$ be a smooth
function which satisfies identities  \eqref{eqbi3} , \eqref{eqbi2} and \eqref{eqbi1}. With respect
to the Levi-Civita connection $\nabla^S$ of the (pseudo)
Riemannian metric given by \eqref{mh} for $\mu=\frac{S}2$, the
function $\phi$ satisfies the following identity
\begin{equation}\label{eqlv1}
(\nabla^S)^3\phi(A,B,C)+Sd\phi(A)h^S(B,C)+\frac{S}2d\phi(B)h^S(C,A)+\frac{S}2d\phi(C)h^S(A,B)=0,\quad A,B,C\in \Gamma(TM).
\end{equation}
\end{prop}
\begin{proof}
From \eqref{eqbi3}, the properties of the Biquard connection, the Ricci identities, the vanishing of the torsion of the Biquard connection  and the integrability of the vertical space we have  the following  equalities
\begin{eqnarray}\label{verhes}
 0=\nabla^2\phi(X,\xi_s)=\nabla^2\phi(\xi_s,X)=\nabla^2\phi(\xi_s,\xi_t), \\ \nabla^2\phi(X,Y)-\nabla^2\phi(Y,X)=2\sum_{s=1}^3\omega_s(X,Y)d\phi(\xi_s)=0.
\end{eqnarray}
Next, using the equality \eqref{nabla-lambda} together with the Ricci identities for the Levi-Civita connection, \eqref{verhes} gives the identities
\begin{eqnarray}\label{lc1}
(\nabla^S)^2\phi(Y,X)=(\nabla^S)^2\phi(X,Y)=\nabla^2\phi(X,Y)-d\phi(L(X,Y))=\nabla^2\phi(X,Y).\\\label{lc2} (\nabla^S)^2\phi(X,\xi_s)=(\nabla^S)^2\phi(\xi_s,X)=\nabla^2\phi(\xi_s,X)-d\phi(L(\xi_s,X)=-Sd\phi(I_sX);\\\label{lc3} (\nabla^S)^2\phi(\xi_s,\xi_t)=(\nabla^S)^2\phi(\xi_t,\xi_s)=\nabla^2\phi(\xi_s,\xi_t)-d\phi(L(\xi_s,\xi_t)=0.
\end{eqnarray}
Now we turn to the computation of the third derivative. Using \eqref{lc1} and \eqref{lc2} we obtain the identities
\begin{multline}\label{lc4}
(\nabla^S)^3\phi(X,Y,Z)=\nabla^3\phi(X,Y,Z)-(\nabla^S)^2\phi(L(X,Y),Z)-(\nabla^S)^2\phi(Y,L(X,Z))\\=\nabla^3\phi(X,Y,Z)+\sum_{s=1}^3\Big[\omega_s(X,Y)(\nabla^S)^2\phi(\xi_s,Z)+\omega_s(X,Z)(\nabla^S)^2\phi(Y,\xi_s)\Big]\\=\nabla^3\phi(X,Y,Z)-\frac{S}2\sum_{s=1}^3\Big[\omega_s(X,Y)d\phi(I_sZ)+\omega_s(X,Z)d\phi(I_sY)\Big]\\=Sdf(X)g(Y,Z)+\frac{S}2df(Y)g(Z,X)+\frac{S}2df(Z)g(X,Y),
\end{multline}
where we used \eqref{eqbi1} in the last equality. Proceeding in the same fashion, we obtain
\begin{multline}\label{lc5}
(\nabla^S)^3\phi(\xi_s,Y,Z)=\nabla^3\phi(\xi_s,Y,Z)-(\nabla^S)^2\phi(L(\xi_s,Y),Z)-\frac{S}2(\nabla^S)^2\phi(Y,L(\xi_s,Z))\\=0 -\frac{S}2(\nabla^S)^2\phi(I_sY,Z)-(\nabla^S)^2\phi(Y,I_sZ)=-\frac{S}2\nabla^2\phi(I_sY,Z)-\frac{S}2\nabla^2\phi(Y,I_sZ)=0,
\end{multline}
where we used \eqref{eqbi2} in the last equality. A similar computation shows
\begin{multline}\label{lc6}
(\nabla^S)^3\phi(Y,Z,\xi_s)=(\nabla^S)^3\phi(Y,\xi_s,Z)\\=\nabla(\nabla^S)^2(Y,\xi,Z)-(\nabla^S)^2\phi(L(Y,\xi_s),Z)-S(\nabla^S)^2\phi(\xi,L(Y,Z))\\=-\frac{S}2\nabla^2\phi(Y,I_sZ) -\frac{S}2(\nabla^S)^2\phi(I_sY,Z)=-\frac{S}2\nabla^2\phi(I_sY,Z)-\frac{S}2\nabla^2\phi(Y,I_sZ)=0,
\end{multline}
where we used \eqref{eqbi2} in  the last equality, and also
\begin{gather}\label{lc7}
(\nabla^S)^3\phi(Y,\xi_s,\xi_s)=\nabla^3\phi(Y,\xi_s,\xi_s)-2(\nabla^S)^2\phi(L(Y,\xi_s)\xi_s)=-2\nabla^2\phi(I_sY,\xi_s)=-Sd\phi(Y);\\\label{lc8}
(\nabla^S)^3\phi(Y,\xi_s,\xi_t)=\nabla^3\phi(Y,\xi_s,\xi_t)-(\nabla^S)^2\phi(L(Y,\xi_s),\xi_t)-(\nabla^S)^2\phi(\xi_s,L(Y,\xi_t))=0.
\end{gather}
Finally, we calculate
\begin{gather}\label{lc9}
(\nabla^S)^3\phi(\xi_s,\xi_s,Y)=\nabla^3\phi(\xi_s,\xi_s,Y)-(\nabla^S)^2\phi((L(\xi_s,\xi_s),Y)-(\nabla^S)^2\phi(\xi_s,L(\xi_s,Y))=-\frac{S}2d\phi(Y);\\\label{lc10}
(\nabla^S)^3\phi(\xi_s,\xi_t,Y)=\nabla^3\phi(\xi_s,\xi_t,Y)-(\nabla^S)^2\phi(L(\xi_s,\xi_t),Y)-(\nabla^S)^2\phi(\xi_s,L(\xi_t,Y))\\=-\frac{S}2d\phi(I_sI_tY)-\frac{S}2d\phi(I_tI_sY)=0.
\end{gather}
Equations \eqref{lc1}-\eqref{lc10} show the validity of  \eqref{eqlv1} for all $A,B,C\in\Gamma(TM)$. This completes the proof of the Proposition.
\end{proof}

\section{Compact qc-hypersurfaces}
\subsection{Proof of Theorem~\ref{t:main1}}
\begin{proof}
We begin by showing that if a function $\phi$ satisfies
\eqref{eqbi1}, then $h\overset{def}{=}\triangle \phi$ is
necessarily an eigenfunction for the sub-Laplacian $\triangle h
=tr^g(\nabla^2 h)$.  Indeed, see \cite[(2.7)]{Tan} for the analogous calculation in the Riemannian case, taking a trace in \eqref{eqbi1} we obtain that $X(\triangle \phi)=-4(n+1)Sd\phi(X)$
which yields $\nabla^2\triangle \phi(X,Y)=-4(n+1)S\,\nabla^2\phi
(X,Y)$ and $\triangle h=- 4(n+1)Sh.$  Since $M$ is compact  it follows $S\geq 0$.

If  the qc-scalar curvature vanishes, $S=0$, then  it follows $\phi=const$, which contradicts our assumption that $M$ is non-umbilic, see Corollary \ref{l:umbilic}.  Thus, we have $S>0$.  In fact, after a qc-homothety, we can assume that $S=2$. Let $h\overset{def}{=}h^S$ be the corresponding Riemannian metric on $M$. Now, in view of \eqref{eqlv1},  by Gallot-Obata-Tanno's theorem \cite{Tan,Galo,Obata}  it follows that the Riemannian manifold $(M, h)$ is  isometric to the round sphere of radius $1$.
Therefore, the curvature tensor $R^h$ of the Levi-Civita connection $\nabla^h$ of $h$ is given by
\begin{equation}\label{riem}
R^h(A,B,C,D)=h(B,C)h(A,B)-h(B,D)h(A,C).
\end{equation}
 The relation between the curvatures of the Levi-Civita connection and the Biquard connection for qc-Einstein spaces with $S=2$ (i.e., 3-Sasakian spaces) \cite[Corollary~4.13]{IMV1} or \cite[Theorem~4.4.3]{IV3} together with \eqref{riem} yields
\begin{multline}\label{biqcurv}
R(X,Y,Z,W)=R^h(X,Y,Z,W)\\+\sum_{s=1}^3\Big[\omega_s(Y,Z)\omega_s(X,W)-\omega_s(X,Z)\omega_s(Y,W)-2\omega_s(X,Y)\omega_s(Z,W)\Big]\\
=h(Y,Z)h(X,W)-h(Y,W)h(X,Z)+\sum_{s=1}^3\Big[\omega_s(Y,Z)\omega_s(X,W)-\omega_s(X,Z)\omega_s(Y,W)-2\omega_s(X,Y)\omega_s(Z,W)\Big].
\end{multline}

According to
\cite[Proposition~4.2]{IV1}, the qc conformal curvature tensor
$W^{qc}$ can by expressed in terms of the curvature $R$ of the
Biquard connection, in general, on a qc-Einstein spaces with qc
scalar curvature $S$ by the formula
\begin{multline}\label{wqc}
W^{qc}(X,Y,Z,W)=R(X,Y,Z,W)+\frac{S}2\Big\{-g(X,W)g(Y,Z)+g(X,Z)g(Y,W)+\\
\sum_{s=1}^3\Big[-\omega_s(X,W)\omega_s(Y,Z)+\omega_s(X,Z)\omega_s(Y,W)+2\omega_s(X,Y)\omega_s(Z,W)\Big]\Big\}.
\end{multline}
Then, since in our case $g(X,Y)=h(X,Y)$ and $S=2$, \eqref{biqcurv} implies that $W^{qc}=0$
and therefore, $(M,\eta)$ is qc-conformally flat (cf. \cite[Theorem~1.3]{IV1}).
Now, the result follows by Theorem \ref{qcc}.

Let us remark that the final step of the proof is similar to an argumentation that had been already used before in the proof of \cite[Theorem 1.3]{IPV}.

\end{proof}

 In the case of a positive qc-scalar curvature of the calibrated qc structure we can substitute the compactness with completeness assumption of the Riemannian metric noting that the Gallot-Obata-Tanno's theorem holds for  a
complete Riemannian manifold. In particular, the manifold is compact. In addition, the local qc-conformal maps considered in the proof of Theorem \ref{t:main1} define a global qc-conformality to the round sphere, see Theorem \ref{qcc}. Therefore, we have
\begin{thrm}\label{main12}
Let $M$ be a simply connected qc hypersurface of a hyper-K\"ahler
manifold which is not totally umbilical. Suppose that the
calibrated qc structure $(\eta_1,\eta_2,\eta_3)$ on $M$ has a
positive qc-scalar curvature and that it is complete with respect
to the natural Riemannian metric $h=g+\eta_1^2+\eta_2^2+\eta_3^2$.
Then the calibrated qc structure on $M$ is  qc-homothetic to the standard 3-Sasakian sphere.
\end{thrm}

\section{Locally embedded qc-hypersurfaces}\label{s:loc emb}

In the non-compact case we show
\begin{thrm}\label{main2}
Let $M$ be a qc-hypersurface in a hyper-K\"ahler manifold such  that all  points of $M$ are non-umbilic. Then there exists a 7 dimensional involutive distribution
$\mathcal D$ on $M$ such that the induced qc structure on  each
integral leaf of $\mathcal D$ is locally  qc-conformal to the
standard 7-dimensional 3-Sasakian sphere.
\end{thrm}

\begin{proof}
We achieve Theorem~\ref{main2} with a series of lemmas. We begin
with the following

\begin{lemma}\label{integr} Let $M$ be a qc Einstein space with local qc 1-forms $\eta_1,\eta_2,\eta_3$
satisfying the structure  equations \eqref{str_eq_mod} and let $\xi_1,\xi_2,\xi_3$ be the corresponding
Reeb vector fields. If there exists a function $\phi$ with a nowhere vanishing horizontal gradient $\nabla\phi$ on $M$, satisfying
\eqref{eqbi1}-\eqref{eqbi3}, then the 7-dimensional distribution
$D=span\{\xi_1,\xi_2,\xi_3,\nabla\phi,I_1\nabla\phi,I_2\nabla\phi,I_3\nabla\phi\}$
is integrable.
\end{lemma}
\begin{proof}
Since $\eta_1,\eta_2,\eta_3$ satisfy \eqref{str_eq_mod}, the
vertical distribution $\text{spand}\{\xi_1,\xi_2,\xi_3\}$ is
integrable and we have $\nabla_X\xi_s=0$. Moreover,
\begin{multline}\label{komutator-phi}[\nabla\phi,I_i\nabla\phi]=\nabla_{\nabla\phi}(I_i\nabla\phi)-\nabla_{I_i\nabla\phi}(\nabla\phi)-T(\nabla\phi,I_i\nabla\phi)\\
=-I_i\nabla_{\nabla\phi}(\nabla\phi)-\nabla_{I_i\nabla\phi}(\nabla\phi)-2\sum_{t=1}^3\omega_t(\nabla\phi,I_i\nabla\phi)\xi_t
\overset{\eqref{eqbi2}}=-2g(\nabla\phi,\nabla\phi)\xi_i.
\end{multline}

We have also that $T({\xi_s},X)=0$, which leads to
\[[\xi_s,\nabla\phi]=\nabla_{\xi_s}\nabla\phi-\nabla_{\nabla\phi}\xi_s-T(\xi_s,\nabla\phi)=\nabla^2\phi(\xi_s,e_a)e_a-\nabla_{\nabla\phi}\xi_s\overset{\eqref{eqbi3}}=\nabla_{\nabla\phi}\xi_s\subset D.
\]
Similarly,  $[\xi_s,I_t\nabla\phi]\subset D$ and thus the
integrability of the distribution $D$ is proved.
\end{proof}

We need the following
\begin{lemma}\label{propqc}
The qc-conformal curvature of a qc-Einstein space has the property
\[
W^{qc}(X,Y,Z,U)=W^{qc}(Z,U,X,Y)=W^{qc}(X,Y,I_sZ,I_sU)=W^{qc}(I_sX,I_sY,Z,U).
\]
\end{lemma}
\begin{proof}
The first equality in the lemma is already known, see e.g.
\cite{IMV4}. 
The second equality follows after a small calculation
using formula \eqref{wqc} combined with
\begin{equation}\label{rb1}
\rho_s=-S\omega_s,\quad
R(X,Y,Z,W)=R(Z,W,X,Y)
\end{equation}
 (cf. \cite[(3.28)]{IMV4} and \cite[Theorem 3.1]{IV1}).   
\end{proof}

We proceed with
\begin{lemma}\label{dim7-phi}Let $M$ be a 7-dimensional qc Einstein space with local qc 1-forms $\eta_1,\eta_2,\eta_3$, satisfying the structure  equations \eqref{str_eq_mod}, corresponding Reeb vector fields $\xi_1,\xi_2,\xi_3$ and Biquard connection $\nabla$. If there exists a function $\phi$ on $M$
satisfying at each point: (i) $\nabla\phi\ne 0$ , (ii) $d\phi(\xi_1)=d\phi(\xi_2)=d\phi(\xi_3)=0$ and (iii) $\nabla^2\phi(X,Y)=h g(X,Y)$, for a smooth function $h$ on $M$,
then $M$ is locally qc-conformally flat.
\end{lemma}

\begin{proof}
Since we assume that the qc 1-forms $\eta_s$ satisfy \eqref{str_eq_mod}, we have  $\nabla_X\xi_s=0$ and thus
\begin{equation}\label{n2phixiX}
\nabla^2\phi(\xi_s,X)=\nabla^2\phi(X,\xi_s)=X\Big(d\phi(\xi_s)\Big)=0.
\end{equation}

By differentiating {\it (iii)} we get
\begin{equation}\label{n3phi}
\nabla^3\phi(X,Y,Z)=dh(X)g(Y,Z).
\end{equation}

The Ricci identity for the Biquard connection $\nabla$ implies that
\begin{equation*}
\nabla^2_{X,Y}\nabla \phi-\nabla^2_{Y,X}\nabla\phi=R(X,Y)\nabla\phi - \nabla_{T(X,Y)}\nabla\phi\\=R(X,Y)\nabla\phi-2\omega_s(X,Y)
\nabla_{\xi_s}\nabla\phi\overset{\eqref{n2phixiX}}=R(X,Y)\nabla\phi,
\end{equation*}
which by means of \eqref{n3phi} gives
\begin{equation}\label{Ricci-id}
R(X,Y,Z,\nabla \phi) = -\nabla^3\phi(X,Y,Z)+\nabla^3\phi(Y,X,Z)=-dh(X)g(Y,Z)+dh(Y)g(X,Z).
\end{equation}

We take a trace in \eqref{Ricci-id} to obtain
\begin{equation}
Ric(X,\nabla\phi)=-3dh(X).
\end{equation}

On the other hand, since $M$ is qc Einstein, $Ric(X,Y)=6Sg(X,Y)$, hence $Ric(X,\nabla \phi)=6Sd\phi(X)$. Therefore,
\begin{equation}\label{2Sphih}
2Sd\phi(X)+dh(X)=0.
\end{equation}
The qc-conformal curvature tensor is given by \eqref{wqc}, which,
by \eqref{Ricci-id} and \eqref{2Sphih}, implies that
\begin{multline}\label{wqcf1}
W^{qc}(X,Y,Z,\nabla \phi)=R(X,Y,Z,\nabla\phi)+2S\Big(-d\phi(X)g(Y,Z)+d\phi(Y)g(X,Z)\Big)\\
-\Big(2Sd\phi(X)+dh(X)\Big)g(Y,Z)+\Big(2Sd\phi(Y)+dh(Y)\Big)g(X,Z)=0.
\end{multline}

Since the dimension of $M$ is seven and since by assumption
$\nabla \phi\ne 0$ on $M$, the vector fields
$\nabla\phi,I_1\nabla\phi,I_2\nabla\phi,I_3\nabla\phi$ form an
orthogonal frame of the 4-dimensional horizontal distribution $H$.
Then, by \eqref{wqcf1} and Lemma  ~\ref{propqc}, we have
$W^{qc}(X,Y,Z,I_s\nabla\phi)=-W^{qc}(X,Y,I_sZ,\nabla\phi)=0$ which
implies that $W^{qc}(X,Y,Z,W)=0$, i.e. $M$ is locally qc
conformally flat.

\end{proof}

The next lemma together with \cite[Theorem~3.1]{IV1} completes
the proof of Theorem~\ref{main2}.

\begin{lemma} Let $M$ be a qc Einstein space, $\phi$ be the
non-constant function satisfying \eqref{eqbi1}-\eqref{eqbi3} and
$D=span\{\xi_1,\xi_2,\xi_3,\nabla\phi,I_1\nabla\phi,I_2\nabla\phi,I_3\nabla\phi\}$
be the integrable distribution from Lemma~\ref{integr}. Then
each integral manifold $\iota:N\rightarrow M$ of $D$ caries an
induced qc structure,  defined locally by the 1-forms
$\iota^*(\eta_1), \iota^*(\eta_2),\iota^*(\eta_3)$, which is qc
conformally flat and qc-Einstein with  qc-scalar curvature with
the same sign as the qc-scalar curvature of $M$.
\end{lemma}
\begin{proof}
Let $\mathfrak j:N\rightarrow M$ be any integral manifold of $D$.
Then the pull-back 1-forms $\mathfrak j^*(\eta_1),\mathfrak
j^*(\eta_2),\mathfrak j^*(\eta_3)$ on $N$ define a qc structure on
$N$ with Reeb vector fields $\tilde\xi_s=\mathfrak
j_*^{-1}(\xi_s)$. The horizontal distribution on $N$ is then just
$\tilde H =\mathfrak j_*^{-1}(H)$ and the corresponding
quaternionic structure on it is given by the endomorphisms
$\tilde I_s=\mathfrak j_*^{-1}I_s\mathfrak j_*$. Moreover,  the
pull-backs of the structure equations \eqref{str_eq_mod} remain
satisfied on $N$ and thus the induced qc structure on $N$ is again
qc Einstein with the same qc scalar curvature as $M$. Let us
denote the corresponding Biquard connection on $N$ by
$\tilde\nabla$ and consider the function $\tilde\phi=\mathfrak j
^*\phi$. Then, clearly, $\tilde\nabla\ (\tilde\phi)=\mathfrak j
_*^{-1}\nabla\phi$ and thus, for any $s=1,2,3$,
\begin{equation*}
[\tilde\nabla\tilde\phi,\tilde I_s\tilde\nabla\tilde\phi] =\mathfrak j_*^{-1}[\nabla\phi,I_s\nabla\phi]\overset{\eqref{komutator-phi}}=\mathfrak j_*^{-1}\Big(-2g(\nabla\phi,\nabla\phi) \xi_s\Big)=-2\tilde g(\tilde\nabla\tilde\phi,\tilde\nabla\tilde\phi)\tilde\xi_s.
\end{equation*}
Therefore,
\begin{equation*}
-2\tilde g(\tilde\nabla\tilde\phi,\tilde\nabla\tilde\phi)\tilde\xi_s=[\tilde\nabla\tilde\phi,\tilde I_s\tilde\nabla\tilde\phi]=\tilde I_s(\tilde\nabla_{\tilde\nabla\tilde\phi}\tilde\nabla\tilde\phi)
-\tilde\nabla_{\tilde I_s\tilde\nabla\tilde\phi}\tilde\nabla\tilde\phi - 2\sum_{t=1}^3\tilde\omega_t(\tilde\nabla\tilde\phi,\tilde I_s\tilde\nabla\tilde\phi)\tilde\xi_t,
\end{equation*}
i.e. we have
\begin{equation*}
\tilde\nabla^2\tilde\phi(\tilde\nabla\tilde\phi,\tilde I_sX)=-\tilde\nabla^2\tilde\phi(\tilde I_s\tilde\nabla\tilde\phi,X),\qquad X\in \tilde H.
\end{equation*}
Since the four vector fields $\tilde\nabla\tilde\phi,\tilde I_1\tilde\nabla\tilde\phi, \tilde I_2\tilde\nabla\tilde\phi,\tilde I_3\tilde\nabla\tilde\phi$ define a frame for the distribution $\tilde H$ we obtain that
\begin{equation*}
\tilde\nabla^2 \phi (\tilde I_sX,\tilde I_sY)=\tilde\nabla^2 \phi (X,Y)
\end{equation*}
for any $X,Y\in\tilde H$ and $s=1,2,3$. This implies that
$\tilde\nabla^2\tilde\phi(X,Y)=h\tilde g(X,Y)$ and thus the
function $\tilde\phi$ satisfies the assertions of
Lemma~\ref{dim7-phi}. Therefore, the integral manifold $N$ is
locally qc-conformally flat.
\end{proof}

\end{proof}

We finish the section with the prof of Theorem \ref{t:main2}.
\begin{proof}
The proof is similar to that of Theorem \ref{t:main1} noting that, here, the qc-conformal flatness follows from Lemma \ref{dim7-phi}. However, the (constant) qc-scalar curvature is not necessarily positive. The proof is complete taking into account
Theorem~\ref{qcc}.
\end{proof}

\section{Appendix. }
In the course of the paper we used several times the fact that a qc-Einstein qc-conformally flat manifold is locally qc-homothetic to one of the standard model qc-spaces \eqref{e:model spaces}.  As indicated below, this fact has been essentially proved before, but due to its independent interest we formulate it explicitly. Furthermore, we include an argument for global equivalence.

\begin{thrm}\label{qcc}
A qc-conformally flat qc-Einstein manifold $M$ is
locally qc-homothetic to one of the following three model spaces:
the 3-Sasakian sphere $S^{4n+3}$, the quaternionic Heisenberg
group $\QH$ or the hyperboloid $S_3^{4n}$  depending on the sign of the qc-scalar curvature, respectively. If in addition $M$
is connected, simply connected with complete Biquard connection
then we have a global qc-homothety  with the model spaces \eqref{e:model spaces}.
\end{thrm}

\begin{proof}By a qc-homothety, depending on the sign of the qc-scalar curvature, we can reduce the claim to one of the cases $S=2$, $S=0$ or $S=-2$.
We recall  that the model spaces \eqref{e:model spaces} are qc-Einstein qc-conformally flat manifolds  with
positive qc-scalar curvature $S=2$ in the case {\textit{i)}} of the 3-Sasakian sphere \cite{IMV1,IPV},  flat in the case of the
quaternionic Heisenberg group  \textit{iii)}
\cite{IMV1}, and negative qc-scalar
curvature $S=-2$,  \cite{IMV5}, for the hyperboloid  \textit{ii)}.

One proof  of the local equivalence goes as follows.
Due to the local qc-conformality with the quaternionic Heisenberg group, with the help of \cite[Theorem 6.2]{IV14}, see \cite[Theorem 1.2]{IMV1} for the positive qc-scalar curvature case,  we can determine the exact form of the conformal factor relating the invariant qc structure on the Heisenberg group to the image by a qc-conformal transformation of the  given  qc-Einstein structure.  The proof of the local equivalence statement in Theorem~\ref{t:main1} follows, for more details see \cite[Theorem 1.2]{IMV1} in the case of positive qc-scalar curvature, the paragraph after \cite[Lemma 8.6]{IV14} in the zero qc-scalar curvature case, while the negative qc-scalar curvature case follows analogously.
The global result in the case of a compact manifold is achieved by a monodromy argument and Liouville's theorem \cite[Theorem 8.5]{IV14}, \cite{CS09}. Below is an argument  using that in our case Biquard's connection is an affine connection with parallel torsion and parallel curvature, hence we can invoke the results in  \cite[Chapter VI]{KN1}.

For a qc-Einstein manifold we have from \cite{IMV1,IMV4}
$T^0=U=0$, the qc-scalar curvature is constant, $S=const$ and the
vertical space is integrable. As a consequence, on a qc-Einstein
manifold we have \cite{IMV1,IV1,IV3,IMV4}
\begin{eqnarray}\label{rho}
T(X,Y)=2\sum_{s=1}^3\omega_s(X,Y)\xi_s; \quad
T(\xi_i,\xi_j)=-S\xi_k, \\
\label{rb2}
R(\xi_s,X,Y,Z)=R(\xi_s,\xi_t,X,Y)=0,
\quad R(A,B)\xi=-2S\sum_{s=1}^3\omega_s(A,B)\xi_s\times\xi.
\end{eqnarray}
{ Using \eqref{der}, we obtain from \eqref{rho} that the torsion
of the Biquard connection is parallel, $\nabla T=0$. Similarly,
\eqref{rb2} implies
that $\nabla R(\xi_s,A,B,C)=\nabla
R(A,B,C,\xi_s)=0$.}

For the horizontal part of $R$ we apply the second condition of the qc-conformal flatness,
$W^{qc}=0$. A substitution of  \eqref{rb1} into  \eqref{wqc} gives
\begin{multline}\label{wqc0}
R(X,Y,Z,W)=\frac{S}2\Big[g(Y,Z)g(X,W)-g(Y,W)g(X,Z)\Big]\\+\frac{S}2\sum_{s=1}^3\Big[\omega_s(Y,Z)\omega_s(X,W)-\omega_s(X,Z)\omega_s(Y,W)-2\omega_s(X,Y)\omega_s(Z,W)\Big].
\end{multline}
Hence, by \eqref{der}, it follows that the horizontal curvature of the Biquard connection is  parallel as well, i.e., we have $\nabla T=\nabla R=0$.

Let $F$ be a
linear isomorphism between the tangent spaces $T_p(M)$ and $T_p'(M')$ of a point $p$ in $M$ and a point $p'$ in the model space \eqref{e:model spaces} of same qc-scalar curvature, such that, $F$ maps an orthonormal basis $\{e_a,I_1 e_a, I_2 e_a, I_3 e_a\}_{a=1}^n$ of the horizontal space at $p$ to the an orthonormal basis $\{e'_a,I'_1 e'_a, I'_2 e'_a, I'_3 e_a\}_{a=1}^n$ of the horizontal space at $p'$ and also sends the corresponding Reeb vector fields at $p$ to those at $p'$.
Thus, $F$ preserves the horizontal and vertical spaces
$F(H_p)=H'_q, \quad F(V_p)=V'_q$, and the $Sp(n)Sp(1)$-structure,
i.e., it maps the tensors $g_p,(I_s)|_p, (\xi_s)|_p$ at the point
$p\in M$ into the tensors $g'_q,(I'_s)|_q, (\xi'_s)|_q$. Taking into account $S=S'$, \eqref{rho} together with \eqref{rb2}, and \eqref{wqc0} show that $F$ maps the
torsion $T_p$ and the curvature $R_p$ at $p$ into the torsion
$T'_q$ and the curvature $R'_q$ at $q\in M'$, respectively.

Now, we can apply the affine equivalence theorem \cite[Theorem~7.4]{KN1} to obtain an affine local
isomorphism between $M$ and the coresponding model space. Since the qc structure $(H\oplus V,\mathbb{Q}, g)$ is
parallel the affine local isomorphism  is a qc-homothety.

Finally, if in addition $M$ is connected, simply connected with a
complete Biquard connections then \cite[Theorem~7.8]{KN1}
gives us a global qc-homothety to the corresponding model case.
We note that the Biquard connection in each of the model cases is complete since the 3-Sasakian spere is compact, the Biquard connection on the qc Heisenberg group  is an invariant connection of a homogeneous space, while the hyperboloid is $Sp(n,1)Sp(1)/Sp(n)Sp(1)$, see e.g. \cite[Theorem~5.1]{AK}, with
the invariant Biquard connection determined by
\eqref{nabla-lambda}.
\end{proof}

\end{document}